\documentclass[11pt,a4paper]{article}
\usepackage[latin1]{inputenc}
\usepackage{amsmath}
\usepackage{amsthm}
\usepackage{amsfonts}
\usepackage{amssymb,url}
\usepackage{graphicx}
\usepackage[margin=2.7cm]{geometry}
\usepackage{xcolor}
\newtheorem{theorem}{Theorem}
\newtheorem{lemma}{Lemma}

\author{Herbert Batte$ ^{1} $, Taboka P. Chalebgwa$ ^{2} $, and Mahadi Ddamulira$  ^{1,3,*}$}
\title{Perrin numbers that are concatenations of two distinct repdigits}
\date{}

\begin{document}
\maketitle

\begin{abstract}
\noindent Let $ (P_n)_{n\ge 0}$ be the sequence of Perrin numbers defined by ternary relation $ P_0=3 $, $ P_1=0 $, $ P_2=2 $, and $ P_{n+3}=P_{n+1}+P_n $ for all $ n\ge 0 $. In this paper, we use Baker's theory for nonzero linear forms in logarithms of algebraic numbers and the reduction procedure involving the theory of continued fractions, to explicitly determine all Perrin numbers that are concatenations of two distinct repeated digit numbers.
\end{abstract}



\noindent 
\thanks{$ ^{*} $ Corresponding author}

\section{Introduction.}

\noindent Let $(P_n)_{n \ge 0}$ be the sequence of Perrin numbers, given by the ternary recurrence relation
\begin{align*}
P_{n+3} = P_{n+1} + P_{n}, \quad \text{for $n \geq 0$, with the initial conditions $P_0 = 3$, $P_1 = 0, $ and $P_2 = 2$}.
\end{align*}
The first few terms of this sequence are
\begin{align*}
(P_n)_{n \geq 0} = \{3, 0, 2, 3, 2, 5, 5, 7, 10, 12, 17, 22, 29, 39, 51, 68, 90, 119, 158, 209, 277, 367, 486, \ldots\}.
\end{align*}
A \textit{repdigit} (in base $10$) is a non negative integer $N$ that has only one distinct digit. That is, the decimal expansion of $N$ takes the form
\begin{align*}
N = \overline{\underbrace{d\cdots d}_\text{$\ell$ times}} = d \left( \frac{10^\ell - 1}{9} \right),
\end{align*}
for some non negative integers $d$ and $\ell$ with $0 \leq d \leq 9$ and $l \geq 1$. This paper is an addition to the growing literature around the study of Diophantine properties of certain linear recurrence sequences. More specifically, our paper focuses on a Diophantine equation involving the Perrin numbers and repdigits. This is a variation on a theme on the analogous problem for the Padovan numbers, a program developed in \cite{LH} and \cite{MD2}.
\medskip

\noindent In \cite{LH}, the authors found all repdigits that can be written as a sum of two Padovan numbers. This result was later extended by the third author to repdigits that are a sum of three Padovan numbers in \cite{MD1}. In another direction, in \cite{MD2}, Ddamulira considered all Padovan numbers that can be written as a concatenation of two repdigits and showed that the largest such number is $Pad(21) = 200$. More formally, it was shown that if $Pad(n)$ is a solution of the Diophantine equation $Pad(n) = \overline{\underbrace{d_1\cdots d_1}_\text{$\ell$ times} \underbrace{d_2\cdots d_2}_\text{$m$ times}}$, then
\begin{align*}
Pad(n) \in \{12,16,21,28,37,49,65,86,114,200 \}.
\end{align*}
The Padovan numbers and Perrin numbers share many similar properties. In particular, they have the same recurrence relation, the difference being that the Padovan numbers are initialized via $Pad(0) = 0$ and $Pad(1) = Pad(2) = 1$. This means that the two sequences also have the same characteristic equation.
\medskip

\noindent Despite the similarities, the two sequences also have some stark differences. For instance, the Perrin numbers satisfy the remarkable divisibility property that if $n$ is prime, then $n$ divides $P_n$. One can easily confirm that this does not hold for the Padovan numbers.

\noindent Inspired by the second author's result in \cite{MD2}, we study and completely solve the Diophantine equation:
\begin{align}\label{eq1}
P_n = \overline{\underbrace{d_1\cdots d_1}_\text{$\ell$ times} \underbrace{d_2\cdots d_2}_\text{$m$ times} } = d_1 \left(  \frac{10^\ell - 1}{9} \right) \times 10^m + d_2 \left(  \frac{10^m - 1}{9} \right),
\end{align}
where $d_1\ne d_2 \in \{0, 1, 2, \ldots, 9 \}$, $d_1>0 $,  $\ell, m \ge 1$, and $n \ge 0$.
\medskip

\noindent We ignore the $d_1 = d_2$ case for the time being, since it has been covered within a more general context in an upcoming paper, where we study the reverse question of repdigits which are sums of Perrin numbers. In any case, the only such Perrin number which is a solution of the above Diophantine equation is $P_{11} = 22$.
\medskip

\noindent Our main result is the following.
\begin{theorem}\label{thm1x}
The only Perrin numbers which are concatenations of two distinct repdigits are 
\begin{align*}
P_n \in \{ 10, 12, 17, 29, 39, 51, 68, 90, 119, 277, 644 \}.
\end{align*}

\end{theorem}

\section{Preliminary Results.}
In this section we collect some facts about Perrin numbers and other preliminary lemmas that are crucial to our main argument.

\subsection{Some properties of the Perrin numbers.}
Recall that the characteristic equation of the Perrin sequence is given by $\phi (x) := x^3 - x - 1=0$, with zeros $\alpha$, $\beta$ and $\gamma = \overline{\beta}$ given by:
\begin{align*}
\alpha = \frac{r_1 + r_2}{6} \quad \text{and} \quad \beta = \frac{-(r_1 + r_2) + i \sqrt{3}(r_1 - r_2)}{12},
\end{align*}
where
\begin{align*}
r_1 = \sqrt[3]{108 + 12 \sqrt{69}} \quad \text{and} \quad r_2 = \sqrt[3]{108 - 12 \sqrt{69}}.
\end{align*}
For all $n \geq 0$, Binet's formula for the Perrin sequence tells us that the $n$th Perrin number is given by
\begin{align}\label{eq2}
P_n = \alpha^n + \beta^n + \gamma^n.
\end{align}
Numerically,  the following estimates hold for the quantities $\{\alpha, \beta, \gamma \}$:
\begin{align*}
1.32 < & \alpha < 1.33, \\
0.86 < |\beta| = & |\gamma| = \alpha^{-\frac{1}{2}} < 0.87.
\end{align*}
It follows that the complex conjugate roots $\beta$ and $\gamma$ only have a minor contribution to the right hand side of equation \eqref{eq2}. More specifically, let
\begin{align*}
e(n) := P_n - \alpha^n =  \beta^n + \gamma^n. \quad \text{Then,} \quad |e(n)| < \frac{3}{\alpha^{n/2}} \quad \text{for all} \quad n \geq 1.
\end{align*}
The following estimate also holds:
\begin{lemma}\label{lem1}
Let $n \geq 2$ be a positive integer. Then
\begin{align*}
\alpha^{n-2} \leq P_n \leq \alpha^{n+1}.
\end{align*}
\end{lemma}
\noindent Lemma \ref{lem1} follows from a simple inductive argument, and the fact that $\alpha^3 = \alpha + 1$, from the characteristic polynomial $\phi$.
\medskip

\noindent Let $\mathbb{K} := \mathbb{Q}(\alpha, \beta)$ be the splitting field of the polynomial $\phi$ over $\mathbb{Q}$. Then $[\mathbb{K}: \mathbb{Q}] = 6$ and $[\mathbb{Q}(\alpha): \mathbb{Q}] = 3$. We note that, the Galois group of $\mathbb{K}/\mathbb{Q}$ is given by
\begin{align*}
\mathcal{G} := \text{Gal}(\mathbb{K}/\mathbb{Q}) \cong \{ (1), (\alpha \beta), (\alpha \gamma), (\beta \gamma), (\alpha \beta \gamma)  \} \cong S_3.
\end{align*}
We therefore identify the automorphisms of $\mathcal{G}$ with the permutation group of the zeroes of $\phi$. We highlight the permutation $(\alpha \beta)$, corresponding to the automorphism $\sigma : \alpha \mapsto \beta, \beta \mapsto \alpha, \gamma \mapsto \gamma$, which we use later to obtain a contradiction on the size of the absolute value of a certain bound.

\subsection{Linear forms in logarithms.}
Our approach follows the standard procedure of obtaining bounds for certain linear forms in (nonzero) logarithms. The upper bounds are obtained via a manipulation of the associated Binet's formula for the given sequence. For the lower bounds, we need the celebrated Baker's theorem on linear forms in logarithms. Before stating the result, we need the definition of the (logarithmic) Weil height of an algebraic number.
\medskip

\noindent Let $\eta$ be an algebraic number of degree $d$ with minimal polynomial
\begin{align*}
P(x) = a_0 \prod_{j = 1}^d (x - \alpha_j),
\end{align*}
where the leading coefficient $a_0$ is positive and the $\alpha_j$'s are the conjugates of $\alpha$. The logarithmic height of $\eta$ is given by
\begin{align*}
h(\eta) := \frac{1}{d} \left( \log a_0 + \sum_{j = 1}^d \log \left( \max \{|\alpha_j|, 1  \}  \right)   \right).
\end{align*}
Note that, if $\eta = \frac{p}{q} \in \mathbb{Q}$ is a reduced rational number with $q > 0$, then the above definition reduces to $h(\eta) = \log \max \{ |p|,q \}$. We list some well known properties of the height function below, which we shall subsequently use without reference:
\begin{align*}
h(\eta_1 \pm \eta_2) & \leq h(\eta_1) + h(\eta_2) + \log 2, \\
h(\eta_1 \eta_2 ^{\pm 1}) & \leq h(\eta_1) + h(\eta_2), \\
h(\eta^s) & = |s| h(\eta), \quad (s \in \mathbb{Z}).
\end{align*}
We quote the version of Baker's theorem proved by Bugeaud, Mignotte and Siksek (\cite{BMS}, Theorem 9.4).
\begin{theorem}[Bugeaud, Mignotte, Siksek, \cite{BMS}]\label{thm1}
Let $\eta_1, \ldots, \eta_t$ be positive real algebraic numbers in a real algebraic number field $\mathbb{K} \subset \mathbb{R}$ of degree $D$. Let $b_1, \ldots, b_t$ be nonzero integers such that 
\begin{align*}
\Gamma := \eta_1 ^{b_1} \ldots \eta_t ^{b_t} - 1 \neq 0.
\end{align*}
Then
\begin{align*}
\log | \Gamma| > - 1.4 \times 30^{t+3} \times t^{4.5} \times D^2 (1 + \log D)(1 + \log B)A_1 \ldots A_t,
\end{align*}
where
\begin{align*}
B \geq \max \{|b_1|, \ldots, |b_t| \},
\end{align*}
and
\begin{align*}
A_j \geq \max \{ D h(\eta_j), |\log \eta_j|, 0.16  \}, \quad \text{for all} \quad j = 1, \ldots,t.
\end{align*}
\end{theorem}

\subsection{Reduction procedure.}
The bounds on the variables obtained via Baker's theorem are usually too large for any computational purposes. In order to get further refinements, we use the Baker-Davenport reduction procedure. The variant we apply here is the one due to Dujella and Peth\H{o} (\cite{DP}, Lemma 5a). For a real number $r$, we denote by $\parallel r \parallel$ the quantity $\min \{|r - n| : n \in \mathbb{Z} \}$, the distance from $r$ to the nearest integer.
\begin{lemma}[Dujella, Peth\H o, \cite{DP}]\label{ls}
Let $\kappa \neq 0, A, B$ and $\mu$ be real numbers such that $A > 0$ and $B > 1$. Let $M > 1$ be a positive integer and suppose that $\frac{p}{q}$ is a convergent of the continued fraction expansion of $\kappa$ with $q > 6M$. Let 
\begin{align*}
\varepsilon := \parallel \mu q \parallel - M \parallel \kappa q \parallel.
\end{align*}
If $\varepsilon > 0$, then there is no solution of the inequality
\begin{align*}
0 < |m \kappa - n + \mu| < AB^{-k}
\end{align*}
in positive integers $m,n,k$ with
\begin{align*}
\frac{\log (Aq/\varepsilon)}{\log B} \leq k \quad \text{and} \quad m \leq M.
\end{align*}
\end{lemma}

\noindent
Lemma \ref{ls} cannot be applied when $\mu=0$ (since then $\varepsilon<0$). In this case, we use the following criterion due to Legendre, a well--known result from the theory of Diophantine approximation. For further details, we refer the reader to the books of Cohen \cite{HC1, HC2}.
\begin{lemma}[Legendre, \cite{HC1, HC2}]
\label{lg}
Let $\kappa$ be real number and $x,y$ integers such that
\begin{align*}
\left|\kappa-\frac{x}{y}\right|<\frac{1}{2y^2}.
\end{align*}
Then $x/y=p_k/q_k$ is a convergent of $\kappa$. Furthermore, let  $ M $ and $ N $ be a nonnegative integers such that $ q_N> M $. Then putting $ a(M):=\max\{a_{i}: i=0, 1, 2, \ldots, N\} $, the inequality 
\begin{align*}
\left|\kappa-\frac{x}{y}\right|\ge \frac{1}{(a(M)+2)y^2},
\end{align*}
holds for all pairs $ (x,y) $ of positive integers with $ 0<y<M $.
\end{lemma}

\noindent We will also need the following lemma by G\'{u}zman S\'{a}nchez and Luca (\cite{GSL}, Lemma 7):
\begin{lemma}[G\'{u}zman S\'{a}nchez, Luca, \cite{GSL}]\label{l3}
Let $r \geq 1$ and $H > 0$ be such that $H > (4r^2)^r$ and $H > L/(\log L)^r$. Then
\begin{align*}
L < 2^r H (\log H)^r.
\end{align*}
\end{lemma}

\section{Proof of the Main Result.}
\subsection{The low range.}
We used a computer program in Mathematica to check all the solutions of the Diophantine equation \eqref{eq1} for the parameters $d_1 \neq d_2 \in \{0, \ldots, 9  \}$, $d_1 > 0$ and $1 \leq \ell, m $ and  $1 \leq n \leq 500$. We only found the solutions listed in Theorem \ref{thm1x}. Henceforth, we assume $n > 500$.
\subsection{The initial bound on $n$.}
We note that equation \eqref{eq1} can be rewritten as
\begin{align}\label{eq3}
P_n & = \overline{\underbrace{d_1\cdots d_1}_\text{$\ell$ times} \underbrace{d_2\cdots d_2}_\text{$m$ times} } \nonumber \\
& = \overline{\underbrace{d_1\cdots d_1}_\text{$\ell$ times}} \times 10^m + \overline{ \underbrace{d_2\cdots d_2}_\text{$m$ times}} \nonumber \\
& = \frac{1}{9} \left(d_1 \times 10^{\ell + m} - (d_1 - d_2) \times 10^m - d_2 \right).
\end{align}
The next lemma relates the sizes of $n$ and $\ell + m$.
\begin{lemma}\label{l4}
All solutions of \eqref{eq3} satisfy
\begin{align*}
(\ell + m) \log 10 - 2 < n \log \alpha < (\ell + m) \log 10 + 1.
\end{align*}
\end{lemma}
\begin{proof}
\noindent Recall that $\alpha^{n-2} \leq P_n \leq \alpha^{n+1}$. We note that 
\begin{align*}
\alpha^{n-2} \le P_n < 10^{\ell + m}.
\end{align*}
Taking the logarithm on both sides, we get 
\begin{align*}
n \log \alpha < (\ell + m) \log 10 + 2 \log \alpha.
\end{align*}
Hence $n \log \alpha < (\ell + m) \log 10 + 1$. The lower bound follows via the same technique, upon noting that $10^{\ell + m -1} < P_n \leq \alpha^{n+1}$.
\end{proof}

\noindent We proceed to examine \eqref{eq3} in two different steps as follows.
\medskip

\noindent \textbf{Step 1.} From equations \eqref{eq2} and \eqref{eq3}, we have that
\begin{align*}
9(\alpha^n +  \beta^n + \gamma^n )= d_1 \times 10^{\ell + m} - (d_1-d_2)\times 10^m - d_2.
\end{align*}
Hence,
\begin{align*}
9 \alpha^n  - d_1 \times 10^{\ell + m} = -9e(n) - (d_1-d_2)\times 10^m - d_2.
\end{align*}
Thus, we have that
\begin{align*}
|9 \alpha^n  - d_1 \times 10^{\ell + m}| & = |-9e(n) - (d_1-d_2)\times 10^m -d_2| \\
& \leq 27 \alpha^{-n/2} + 18 \times 10^m  \\
& < 4.6 \times 10^{m+1},
\end{align*}
 where we used the fact that $n > 500$. Dividing both sides by $d_1 \times 10^{\ell + m}$, we get
\begin{align}\label{eq4}
\left | \left( \frac{9}{d_1} \right) \alpha^n \times 10^{-\ell -m} - 1  \right | < \frac{4.6 \times 10^{m+1}}{d_1 \times 10^{\ell +m}} \leq  \frac{46}{10^\ell}.
\end{align}
We let
\begin{align}\label{eqg}
\Gamma_1 := \left( \frac{9}{d_1} \right) \alpha^n \times 10^{-\ell -m} - 1 .
\end{align}
We shall proceed to compare this upper bound on $|\Gamma_1|$ with the lower bound we deduce from Theorem \ref{thm1}. Note that $\Gamma_1 \neq 0$, since this would imply that $ \alpha^n = \frac{10^{\ell + m} \times d_1}{9}$. If this is the case, then applying the automorphism $\sigma$ on both sides of the preceeding equation and taking absolute values, we have that
\begin{align*}
\left| \frac{10^{\ell + m} \times d_1}{9} \right| = |\sigma ( \alpha^n)| = |\beta^n| < 1,
\end{align*}
which is false. We thus have that $\Gamma_1 \neq 0$.
\medskip

\noindent With a view towards applying Theorem \ref{thm1}, we define the following parameters:
\begin{align*}
\eta_1: = \frac{9}{d_1}, \ \eta_2: = \alpha, \ \eta_3: = 10, \ b_1: = 1, \ b_2: = n, \ b_3: = -\ell - m, \ t: = 3.
\end{align*}
Note that, by Lemma \ref{l4} we have that $\ell + m < n$. Thus we take $B = n$. We note that $\mathbb{K}: = \mathbb{Q}(\eta_1, \eta_2, \eta_3) = \mathbb{Q}(\alpha)$. Hence $D: = [\mathbb{K}: \mathbb{Q}] = 3$.
\medskip

\noindent We note that
\begin{align*}
h(\eta_1) = h \left( \frac{9}{d_1} \right) \le 2h(9) = 2\log 9 < 5.
\end{align*}
We also have that $h(\eta_2) = h(\alpha) = \frac{\log \alpha}{3}$ and $h(\eta_3) = \log 10$. Hence, we let
\begin{align*}
A_1: = 15, \ A_2: = \log \alpha, \ A_3: = 3 \log 10.
\end{align*}
Thus, we deduce via Theorem \ref{thm1} that
\begin{align*}
\log |\Gamma_1| &> -1.4 \times 30^6 \times  3^{4.5}\times  3^2\times  (1 + \log 3)(1 + \log n)(15)(\log \alpha)(3 \log 10)\\& > -1.45\times 10^{30}(1 + \log n).
\end{align*}
Comparing the last inequality obtained above with \eqref{eq4}, we get
\begin{align*}
\ell \log 10 - \log 46 < 1.45 \times 10^{30}(1 + \log n).
\end{align*}
Hence,
\begin{align}\label{eq5}
\ell \log 10 <  1.46 \times 10^{30}(1 + \log n).
\end{align}

\noindent \textbf{Step 2.} We rewrite equation \eqref{eq3} as
\begin{align*}
9 \alpha^n  - d_1 \times 10^{\ell + m} + (d_1-d_2)\times 10^m = -9e(n) - d_2.
\end{align*}
That is,
\begin{align*}
9 \alpha^n  - (d_1 \times 10^\ell - (d_1-d_2))\times 10^m = -9e(n) - d_2.
\end{align*}
Hence,
\begin{align*}
|9 \alpha^n  - (d_1 \times 10^l - (d_1-d_2))\times 10^m| & = |-9e(n) - d_2| \\
& \leq \frac{27}{\alpha^{n/2}} + 9  < 36.
\end{align*}
Dividing throughout by $9 \alpha^n$, we have that
\begin{align}\label{eq6}
\left| \left( \frac{d_1 \times 10^\ell - (d_1 - d_2)}{9} \right)\alpha^{-n} \times 10^m - 1 \right| & < \frac{36}{9 \alpha^n}= \frac{4}{\alpha^n}
\end{align}
We put
\begin{align*}
\Gamma_2 :=  \left( \frac{d_1 \times 10^\ell - (d_1 - d_2)}{9} \right)\alpha^{-n} \times 10^m - 1 
\end{align*}
As before, we have that $\Gamma_2 \neq 0$ because this would imply that 
\begin{align*}
\alpha^n = 10^m \times  \left(\frac{d_1 \times 10^\ell - (d_1-d_2)}{9} \right),
\end{align*}
which in turn implies that
\begin{align*}
\left| 10^m \left(\frac{d_1 \times 10^\ell - (d_1-d_2)}{9} \right) \right| = |\sigma( \alpha^n)| = | \beta^n | < 1,
\end{align*}
which is false. In preparation towards applying Theorem \ref{thm1}, we define the following parameters:
\begin{align*}
\eta_1 := \left(\frac{d_1 \times 10^\ell - (d_1-d_2)}{9} \right), \ \eta_2 := \alpha, \ \eta_3: = 10, \ b_1: = 1, \ b_2: = -n, \ b_3 := m, \ t: = 3.
\end{align*}
In order to determine what $A_1$ will be, we need to find the find the maximum of the quantities $h(\eta_1)$ and $|\log \eta_1|$.
\medskip

\noindent We note that
\begin{align*}
h(\eta_1) & = h \left(\frac{d_1 \times 10^\ell - (d_1-d_2)}{9} \right)\\
& \leq  h(9) + \ell h(10) +h(d_1) + h(d_1-d_2) + \log 2\\
& \leq 4 \log 9 + \ell \log 10\\
& < 1.47\times 10^{30}(1 + \log n),
\end{align*}
where, in the last inequality above, we used \eqref{eq5}. On the other hand, we also have that
\begin{align*}
|\log \eta_1| & = \left| \log \left(\frac{d_1 \times 10^\ell - (d_1-d_2)}{9} \right) \right| \\
& \leq  \log 9 + |\log (d_1 \times 10^\ell - (d_1-d_2))|\\
& \leq \log 9 + \log (d_1 \times 10^\ell) + \left| \log \left( 1 - \frac{d_1 - d_2}{d_1 \times 10^\ell}  \right)  \right|\\
& \leq \ell \log 10 + \log d_1 + \log 9 + \frac{|d_1 - d_2|}{d_1 \times 10^\ell} + \frac{1}{2} \left(  \frac{|d_1 - d_2|}{d_1 \times 10^\ell} \right)^2 + \cdots \\
& \leq \ell \log 10 + 2 \log 9 + \frac{1}{10^\ell} + \frac{1}{2 \times 10^{2\ell}} + \cdots\\
& \leq 1.46 \times 10^{30}(1 + \log n) + 2 \log 9 + \frac{1}{10^\ell - 1} \\
& < 1.47 \times 10^{30}(1 + \log n),
\end{align*}
where, in the second last inequality, we used equation \eqref{eq5}. We note that $D \cdot h(\eta_1) > |\log \eta_1|$.
\medskip

\noindent Thus, we put $A_1: = 4.41 \times 10^{30} (1 + \log n)$. We take $A_2: = \log \alpha$ and $A_3: = 3 \log 10$, as defined in \textbf{Step 1}. Similarly, we take $B: = n$. 
\medskip

\noindent Theorem \ref{thm1} then tells us that
\begin{align*}
\log |\Gamma_2| & > -1.4 \times  30^6 \times 3^{4.5} \times 3^2 \times (1 + \log 3)(1 + \log n)(\log \alpha)(3 \log 10)A_1\\
& > -6 \cdot 10^{12}(1 + \log n)A_1~ > ~-3 \times 10^{43}(1 + \log n)^2.
\end{align*}
Comparing the last inequality with \eqref{eq6}, we have that
\begin{align}\label{eq7}
n \log \alpha < 3 \times 10^{43}(1 + \log n)^2 + \log 4.
\end{align}
Thus, we can conclude that
\begin{align*}
n < 1.10 \times 10^{44}(1 + \log n)^2.
\end{align*}
With the notation of Lemma \ref{l3}, we let $r: = 2$, $L: = n$ and $H: = 1.10 \times 10^{44}$ and notice that this data meets the conditions of the lemma. Applying the lemma, we have that
\begin{align*}
n < 2^2 \times 1.1 \times 10^{44} \times (\log 1.1 \times 10^{44})^2.
\end{align*}
After a simplification, we obtain the bound
\begin{align*}
n < 4.6 \times 10^{48}.
\end{align*}
Lemma \ref{l4} then implies that
\begin{align*}
\ell + m < 6.0 \times 10^{47}.
\end{align*}
The following lemma summarizes what we have proved so far:
\begin{lemma}\label{lf}
All solutions to the Diophantine equation \eqref{eq1} satisfy
\begin{align*}
\ell + m < 6.0 \times 10^{47} \quad \text{and} \quad n < 4.6 \times 10^{48}.
\end{align*}
\end{lemma}

\subsection{The reduction procedure.}
We note that the bounds from Lemma \ref{lf} are too large for computational purposes. However, with the help of Lemma \ref{ls}, they can be considerably sharpened. The rest of this section is dedicated towards this goal. We proceed as in \cite{MD2}.

\noindent Using equation \eqref{eqg}, we define the quantity $\Lambda_1$ as
\begin{align*}
\Lambda_1: = - \log (\Gamma_1 + 1) = (\ell + m)\log 10 - n\log \alpha - \log \left(  \frac{9}{d_1} \right).
\end{align*}
 Equation \eqref{eq4} can thus be rewritten as
\begin{align*}
\left| e^{- \Lambda_1} - 1 \right| < \frac{46}{10^\ell}.
\end{align*}
 If $\ell \ge 2$, then the above inequality is bounded above by $\frac{1}{3}$. Recall that if $x$ and $y$ are real numbers such that $|e^x - 1|< y$, then $x < 2y$. We therefore conclude that $|\Lambda_1|< \frac{92}{10^\ell}$. Equivalently,
\begin{align*}
\left|  (\ell + m)\log 10 - n\log \alpha - \log \left(  \frac{9}{d_1} \right) \right| < \frac{92}{10^\ell}.
\end{align*}
 Dividing throughout by $\log \alpha$, we get
\begin{align}\label{kala}
\left| (\ell + m)\frac{\log 10}{\log \alpha} - n + \left( \frac{\log (d_1/9)}{\log \alpha} \right)  \right| < \frac{92}{10^\ell \log \alpha}.
\end{align}
 Towards applying Lemma \ref{ls}, we define the following quantities:
\begin{align*}
\tau: = \frac{\log 10}{\log \alpha}, \quad \mu (d_1): = \frac{\log (d_1/9)}{\log \alpha}, \quad A := \dfrac{92}{\log\alpha}, \quad B: = 10, \quad \text{and} \quad 1 \leq d_1 \leq 8.
\end{align*}
The continued fraction expansion of $\tau$ is given by
\begin{align*}
\tau = [a_0;a_1,a_2, \ldots] = [8;5,3,3,1,5,1,8,4,6,1,4,1,1,1,9,1,4,4,9,1,5,1,1,1,5,1,1,1,2,1,4,\ldots].
\end{align*}
We take $M: = 6 \times 10^{47}$, which, by Lemma \ref{lf}, is an upper bound for $\ell + m$. A computer assisted computation of the convergents of $\tau$ returns the convergent
\begin{align*}
\frac{p}{q} = \frac{p_{106}}{q_{106}} = \frac{177652856036642165557187989663314255133456297895465}{21695574963444524513646677911090250505443859600601}
\end{align*}
 as the first one for which the denominator $q = q_{106} > 3.6 \times 10^{48} = 6M$. Maintaining the notation of Lemma \ref{ls}, we computed $M \| \tau q \|$ and obtained $M \| \tau q \| < 0.0393724$. The smallest (positive) value of $\|\mu q \|$ we obtained satisfies $\|\mu q  \| > 0.0752711$, corresponding to $d_1 =3$. We thus choose $\epsilon  = 0.0358987 < \|\mu q  \| - M \| \tau q \|$. We deduce that
\begin{align*}
\ell \leq \frac{\log (332q/\epsilon)}{\log 10} < 53.
\end{align*}
For the case $ d_1=9 $, we have that $ \mu(d_1)=0 $. In this case we apply Lemma \ref{lg}. The inequality \eqref{kala} can be rewritten as
\begin{align*}
\left|\frac{\log 10}{\log \alpha} - \frac{n}{\ell+m}  \right| < \frac{92}{10^\ell(\ell+m) \log \alpha}<\dfrac{1}{2(\ell+m)^2},
\end{align*}
because $ \ell+m<6\times 10^{47}:=M $. It follows from Lemma \ref{lg} that $ \frac{n}{\ell+m} $  is a convergent of $ \kappa:=\frac{\log 10}{\log\alpha} $. So $ \frac{n}{\ell+m} $ is of the form $ p_k/q_k $ for some $ k=0, 1, 2, \ldots, 106 $. Thus,
\begin{align*}
\dfrac{1}{(a(M)+2)(l+m)^2}\le \left|\dfrac{\log 10}{\log \alpha}-\frac{n}{\ell+m} \right|<\frac{92}{10^\ell(\ell+m) \log \alpha}.
\end{align*}
Since $ a(M)=\max\{a_{k}: k=0, 1, 2, \ldots, 106\}=564 $, we get that
\begin{align*}
l\le \log\left(\dfrac{566\times 92 \times 6\times 10^{47}}{\log\alpha}\right)/\log 10 < 53.
\end{align*}
Thus, $ \ell\le 53 $ in both cases. In the case $ \ell<2 $, we have that $ \ell<2<53 $. Thus, $ \ell\le 53 $ holds in all cases.

\noindent Proceeding, recall that $d_1 \neq d_2\in \{0, \ldots, 9 \}, \quad d_1 > 0$. We now have that $1 \leq \ell \leq 53$. We define
\begin{align*}
\Lambda_2: = \log (\Gamma_2 + 1) = \log \left( \frac{d_1 \times 10^\ell - (d_1 - d_2)}{9} \right) -n \log \alpha + m \log 10.
\end{align*}
We rewrite inequality \eqref{eq6} as
\begin{align*}
\left| e^{\Lambda_2} - 1 \right| < \frac{4}{\alpha^n}.
\end{align*}
Recall that $n > 500$, therefore $\frac{4}{\alpha^2} < \frac{1}{2}$. Hence $|\Lambda_2| < \frac{8}{\alpha^n}$. Equivalently,
\begin{align*}
\left|  m \log 10 - n \log \alpha + \log \left( \frac{d_1 \times 10^\ell - (d_1 - d_2)}{9} \right) \right| < \frac{8}{\alpha^n}.
\end{align*}
Dividing both sides by $\log \alpha$, we have that
\begin{align}\label{kalas}
\left|  m \left( \frac{\log 10}{\log \alpha} \right) - n + \frac{\log ((d_1 \times 10^\ell - (d_1 - d_2))/9)}{\log \alpha} \right| < \frac{8}{\alpha^n\log\alpha}.
\end{align}
Again, we apply Lemma \ref{ls} with the quantities:
\begin{align*}
\kappa:=\frac{\log 10}{\log \alpha}, \quad \mu(d_1,d_2):=\frac{\log ((d_1 \times 10^l - (d_1 - d_2))/9)}{\log \alpha}, \quad A:=\dfrac{8}{\log\alpha}, \quad B:=\alpha.
\end{align*}
We take the same $ \kappa $ and its convergent $ p/q=p_{106}/q_{106} $ as before. Since $ m<l+m<6\times 10^{47} $, we choose $ M:=6\times 10^{47} $ as the upper bound on $ m $. With the help of Mathematica, we get that $ \varepsilon>0.0000542922 $, and thus
\begin{align*}
n\le \dfrac{\log((8/\log\alpha)q/\varepsilon)}{\log\alpha}<454.
\end{align*}
Therefore, we have that $ n\le 454 $.

\noindent
The case $ \ell=1, ~d_1=1, d_2=0 $ leads to $ \mu(d_1,d_2)=0 $. So, in this case we use Lemma \ref{lg}.  The inequality \eqref{kalas} can be rewritten as
\begin{align*}
\left| \frac{\log 10}{\log \alpha}  - \frac{n}{m} \right| < \frac{8}{\alpha^nm\log\alpha}<\dfrac{1}{2m^2},
\end{align*}
because $ m<\ell+m<6\times 10^{47}:=M $. Proceeding along the same lines as before, we have that $ a(M)=564 $ and thus,
\begin{align*}
\dfrac{1}{566m^2}< \left| \frac{\log 10}{\log \alpha}  - \frac{n}{m} \right| < \frac{8}{\alpha^nm\log\alpha}.
\end{align*}
This yields,
\begin{align*}
n\le \log\left(\dfrac{566\times 8\times 6 \times 10^{47}}{\log\alpha}\right)/\log\alpha <431.
\end{align*}
Thus, $ n\le 454 $ in both cases. This contracts our assumption that $ n>500 $. Hence, Theorem \ref{thm1x} is proved. \qed


\section*{Addresses}
$ ^{1} $ Department of Mathematics, Makerere University, P.O. Box 7062 Kampala, Uganda

\noindent 
Email: \url{hbatte91@gmail.com}

\vspace{0.5cm}
\noindent 
$ ^{2} $ Department of Mathematics and Statistics, McMaster University, 1280 Main Street
West, Hamilton, Ontario, L8S 4K1, Canada

\noindent 
Email: \url{chalebgt@mcmaster.ca}

\vspace{0.5cm}
\noindent 
$ ^{1} $ Department of Mathematics, Makerere University, P.O. Box 7062 Kampala, Uganda

\noindent
$ ^{3} $ Max Planck Institute for Software Systems, Campus E1 5, D-66123 Saarbr\"ucken, Germany

\noindent
Email: \url{mahadi.ddamulira@mak.ac.ug; mddamulira@mpi-sws.org}


\begin{thebibliography}{99}


\bibitem{BMS}
Y Bugeaud, M Mignotte, and S Siksek,
\textit{Classical and modular approaches to exponential Diophantine equations I. Fibonacci and Lucas perfect powers},
Annals of Mathematics, (2), {\bf 163}(2):969--1018, 2006.

\bibitem{HC1}
H. Cohen, {\it A Course in Computational Algebraic Number Theory}, Springer, New York, 1993.

\bibitem{HC2}
H. Cohen, {\it Number Theory. Volume I: Tools and Diophantine Equations}, Springer, New York, 2007.

\bibitem{MD1}
M Ddamulira,
\textit{Repdigits as sums of three Padovan numbers},
Bolet\'{i}n de la Sociedad Matem\'{a}tica Mexicana, {\bf 26}(3):247--261, 2020.

\bibitem{MD2}
M Ddamulira,
\textit{Padovan numbers that are concatenations of two distinct repdigits}, Mathematica Slovaca, {\bf 71}(2):275--284, 2021.


\bibitem{DP}
A Dujella and A Peth\H{o},
\textit{A generalization of a theorem of Baker and Davenport}, The
Quarterly Journal of Mathematics, (2), {\bf 49}(195):291--306, 1998.

\bibitem{LH}
A C Garc\'{i}\'{a} Lomel\'{i} and S Hern\'{a}ndez Hern\'{a}ndez,
\textit{Repdigits as sums of two Padovan numbers.}
Journal of Integer Sequences,  {\bf 22}: Art. No. 19.2.3, 2019.


\bibitem{FL}
F Luca,
\textit{Fibonacci and Lucas numbers with only one distinct digit},
Portugaliae Mathematica, {\bf 57}(2), 2000.


\bibitem{GSL}
S G\'{u}zman S\'{a}nchez and F Luca,
\textit{Linear combinations of factorials and $s$-units in a binary recurrence sequence},
Annales Math\'{e}matiques du Qu\'{e}bec, {\bf 38}(2):169--188, 2014.


\end{thebibliography}
\end{document}